\documentclass[11pt]{amsart}

\usepackage{amsmath}
\usepackage{amssymb}
\usepackage{graphicx}

\newtheorem{theorem}{Theorem}[section]

\newtheorem{proposition}[theorem]{Proposition}
\theoremstyle{definition}
\newtheorem{definition}[theorem]{Definition}

\newtheorem{remark}[theorem]{Remark}

\numberwithin{equation}{section}

\title[Additivity of Witten index]{On a conjectured property of the Witten index and an application to Levinson's theorem}
\author{Alan Carey}
\address{School of Mathematics and Applied Statistics, University of Wollongong, NSW, Australia}
\email{alan.carey@gmail.com}

\author{Galina Levitina}
\address{Mathematical
Sciences Institute, Australian National University, Canberra, ACT, Australia}
\email{galina.levitina@anu.edu.au}


\DeclareMathOperator{\tr}{tr}     
\newcommand{\D}{\mathcal{D}}  
\renewcommand{\H}{\mathcal{H}}  
\renewcommand{\L}{\mathcal{L}} 
\newcommand{\N}{\mathbb{N}}   
\newcommand{\R}{\mathbb{R}}   
\newcommand{\Z}{\mathbb{Z}}   
\newcommand{\bsD}{{\boldsymbol{D}}}

\DeclareMathOperator{\iindex}{index}
\DeclareMathOperator{\dom}{dom}

\begin{document}

\begin{abstract}
A few years ago Fritz Gesztesy raised the issue of whether there was a composition rule for the Witten index analogous to that satisfied by Fredholm operators. In this note we prove a result in this direction and provide an application to Levinson's theorem.
\end{abstract}

\dedicatory{To Fritz Gesztesy, our friend and collaborator}

\maketitle

\section{Introduction}

In \cite{Witten}, Witten introduced a number, termed later the Witten index, which counts the difference in the number of bosonic and fermionic zero-energy modes of a Hamiltonian. In  \cite{BGGSS87}, \cite{BGGS87}, \cite{GS88} this  has  been put into a mathematical framework using two different regularisations (see \eqref{def_WI} for the semigroup regularisation). It was shown that the Witten index is a natural substitute for the Fredholm index of an operator when the operator ceases to be Fredholm.

This note is an application of results described in our Lecture Notes volume \cite{CL}. In \cite{CL} we extended the investigation begun in
\cite{Review} on trace formulas for non-Fredholm operators to the general case required by the study of massless Dirac operators in all dimensions that is 
explained in \cite{CGL+16b]}.
While \cite{CL}  is partly a review of previous progress on developing an index theory for 
non-Fredholm operators it contains a proof of the principal trace formula  (see \eqref{ch_principla trace formula_intro_formula}) expressing the Witten index as an integral of a one form on a Banach manifold of perturbations of a fixed unbounded non-Fredholm self-adjoint operator. This fact leads to a composition rule for the Witten index.

\subsection{Motivation}

The product formula for the Fredholm index states that if $T_1$ and $T_2$ are Fredholm operators then $T_1T_2$ is Fredholm and  the index of $T_1T_2$ is the sum: index$( T_1)$ + index$(T_2)$. 
In this note we use our previous work on the `principal trace formula'  (PTF) and `generalised spectral flow'  \cite{CL} to derive an analogous composition rule for the Witten index. This  question
was suggested by Fritz Gesztesy.  The point raised by Fritz is that the Witten index reduces to the Fredholm index when its definition is applied
to Fredholm operators and thus it is reasonable to ask what properties of the latter extend to the former.
 It is already known that although the topological properties of the Fredholm index are  different from those of the Witten index, the Witten index does posses stability properties with respect to additive perturbations  that satisfies an appropriate relative trace-class assumption  \cite{GS88}. Here we focus only on a composition rule for the Witten index. 

We discuss  in the second half of this note the relevance of our result to a recent paper \cite{AR} on Levinson's theorem in scattering theory and its topological interpretation as discovered in \cite{KR_2006},\cite{KR}.
This application emerged from discussions about \cite{AR}  with the authors.
We also illustrate  some other peculiarities of the Witten index including the fact that, in a one dimensional  example, it interpolates in a `continuous' fashion between the integral values of the Fredholm index of  Wiener-Hopf operators. 

We take up this idea in the final Section where we find an operator with fractional Witten index for which the composition rule for the Witten index relates it to the Fredholm index.
This is a special example of our composition property of the Witten index using our earlier work \cite{Review}

The main result of the paper is in Section \ref{sec_WI} where we discuss the principal trace formula and explain the composition rule for the Witten index. A stronger result holds in the special case that is applicable to the scattering theory example in the final Section. This application was suggested by the existence of an index pairing interpretation of Levinson's theorem  discovered in \cite{AR}.

\subsection{An outline of this note}
The exposition starts with an elementary motivating example using (generalised) Toeplitz operators with discontinuous symbol.
We link this in subsection \ref{sub_preview} to our previous work on the Witten index and its connection to the Gohberg-Krein theorem \cite{Gohberg_Krein}, \cite{RD}.
In the case  of Section \ref{sec_Toeplitz} where we have operators on $L^2(\S^1)$ the method allows us to construct a Fredholm problem from a non-Fredholm one arising from a discontinuous symbol.
On $L^2(\mathbb R)$ this example translates to a problem of dealing with a Wiener-Hopf symbol with bad asymptotics at infinity.
This Wiener-Hopf operator has a Witten index.  We use it as an elementary example of the more general results in the final two Sections.

In Section \ref{sec_WI}  we present the composition rule for the Witten index in two instances. The first is quite general but has an additional technical assumption, which can be removed in the one-dimensional setting as required in Section \ref{sec_Levinson}. In the second 
instance where we work with what we call the `commutative case'
we find that the  Witten index behaves in an analogous fashion, under this composition rule, to the Fredholm index.

We now summarise the notations used throughout this note.  For a Hilbert space $\H$ we denote by $\L(\H)$ the algebra of all bounded operator and by $\L_1(\H)$ the ideal of all trace-class operators. The standard trace on $\L_1(\H)$ is denoted by $\tr$. For two (possibly unbounded) operator $A,B$ we denote by $[A,B]$ their commutator $AB-BA$, when it is clear how to deal with the domains of the operators. For an integer $k\in\N$ we denote by $[A,B]^{(k)}$ the $k$-th repeated commutator $[A,\dots,[A,[A,B]]]$.

The Hilbert space of all square-integrable functions on a measure space $(\Omega,\Sigma,\mu)$ is denoted by $L^2(\Omega)$. The space of all essentially bounded functions on $(\Omega,\Sigma,\mu)$ is denoted by $L^\infty(\Omega)$. The notation $C^\infty(\R)$ (respectively, $C_b^\infty(\R)$) stands for the space of all (respectively, bounded) functions on $\R$ that are  differentiable infinitely many times.

\subsection*{Acknowledgements.}  We thank Angus Alexander, Adam Rennie and Serge Richard for many discussions on Levinson's theorem.

\section{An example of a Toeplitz operator with discontinuous symbol}\label{sec_Toeplitz}

In this Section we give an example of a (generalised) Toeplitz operator with discontinuous symbol and discuss how its (Fredholm) index can be computed via a formula for  the Witten index.\footnote{The motivation for this example stems from a communication from Harold Grosse
many years ago in which he asked  about the significance of gauge transformations of classical spinor fields that did not define Toeplitz operators and 
consequently did not extend to symmetries of fermionic quantum field theories. More precisely, he asked about automorphisms of the algebra of the canonical anti-commutation relations that were not implementable in representations of this algebra of physical interest.  
The discussion in this initial Section of our paper suggests a
way forward to further explore this Grosse Question though we will
not go into that here.}
For background on Toeplitz and Wiener-Hopf operators see \cite{RD}. 

Following \cite{DS} for a Hilbert space $\H$, an orthogonal projection $P$ on $\H$, and a bounded operator $A$ on $\H$ an operator of the form $PAP$ is said to be a (generalised) Toeplitz operator. The classical Toeplitz operator $T_a=PM_aP$ is recovered when $\H$ is $L^2(\mathbb{S})$, $P$ is the Hardy projection onto the Hardy space $H^2(\mathbb{S})$ and $M_a$ is the multiplication operator on $L^2(\mathbb{S})$ by a function $a$.
A classical result of Gohberg and Krein  gives a necessary and sufficient conditions for a Toeplitz operator $T_a=PM_aP$ with a continuous symbol $a$ to be Fredholm. We now present an example of a (generalised) Toeplitz operator with discontinuous symbol which is Fredholm. 
We consider $L^2[-\pi,\pi]$ in two ways, as the completion of the continuous periodic functions on the circle, denoted $\mathcal H_1$ and
as the completion of the antiperiodic functions, denoted $\mathcal H_2$. 
 The functions $\{e_n, n\in \mathbb Z\} $ given by $e_n(\theta)=\frac{1}{\sqrt{2\pi}} e^{in\theta}, \theta\in[-\pi,\pi],$ form an orthonormal basis of $\mathcal{H}_1$ and the functions $e_{n+1/2}, n\in \mathbb Z$ with $e_{n+1/2}(\theta)= \frac{1}{\sqrt{2\pi}}e^{i(n+1/2)\theta},\theta\in[-\pi,\pi],$ 
form an orthonormal basis of $\mathcal{H}_2$.

Let $a(\theta)=e^{i\theta/2}, \theta\in [-\pi,\pi]$ and let 
 $M_{a}$ be the operator from $\mathcal{H}_1$ to $\mathcal{H}_2$ acting by multiplication by the function $a$. We denote by $\mathcal{H}$ the Hilbert space of the direct sum $\mathcal H_1\oplus \mathcal H_2$.  Introduce now an operator $M:\mathcal{H}\to \mathcal{H}$ defined by 
$$M=\begin{pmatrix}
0&M_{a}\\M_{a}&0
\end{pmatrix}.$$
Denoting by $P_i$ the Hardy projection on $\mathcal{H}_i, i=1,2$ (defined as the projection onto the span of the $e_n$ for $n\geq 0$ or onto the span of the $e_{n+1/2}$ for $n\geq 0$ respectively) we also introduce the projection $Q:\mathcal{H}\to\mathcal{H}$ by setting 
$$Q=\begin{pmatrix}P_1&0\\0&P_2\end{pmatrix}.$$

Finally, we introduce now the (generalised) Toeplitz operator $QMQ$ we consider for the example of this section. 

\begin{proposition}\label{prop_example}
The Toeplitz operator $QMQ$ is a Fredholm operator on $Q\mathcal{H}$ of index $-1$. 
\end{proposition}
\begin{proof}
We firstly show that $QM^*Q$ is a parametrix for $QMQ$ (on $Q\mathcal{H}$). We have that  
\begin{align*}
(QMQ)(QM^*Q)=\begin{pmatrix}
P_1M_{1/2}P_2M_{1/2}^*P_1&0\\0&P_2M_{1/2}P_1M_{1/2}^*P_2
\end{pmatrix},\\
(QM^*Q)(QMQ)=\begin{pmatrix}
P_1M^*_{1/2}P_2M_{1/2}P_1&0\\0&P_2M^*_{1/2}P_1M_{1/2}P_2
\end{pmatrix}.
\end{align*}
It is easy to check that
\begin{align*}
P_1M_{1/2}P_2M_{1/2}^*P_1&=\sum_{n\geq 1}\langle \cdot, e_n\rangle e_n\\
P_1M_{1/2}^*P_2M_{1/2}P_1&=P_1,\\
P_2M_{1/2}P_1M_{1/2}^*P_2&=P_2M_{1/2}^*P_1M_{1/2}P_2=P_2.\\
\end{align*} 
Therefore, 
$$(QMQ)(QM^*Q)-Q=\begin{pmatrix}
-\langle \cdot, e_0\rangle e_0&0\\0&0
\end{pmatrix},\quad (QM^*Q)(QMQ)-Q=0,$$
proving that $QM^*Q$ is a parametrix for $QMQ$. 

Using now Fedosov's formula for the Fredholm index (see e.g.  \cite[Theorem 2.3]{Murphy}) we conclude that 
$$\iindex(QMQ)=\tr\big((QMQ)(QM^*Q)-(QM^*Q)(QMQ)\big)=-1.$$
\end{proof}

 To connect the above example to the Witten index discussed in Section~\ref{sec_WI} we shall need to translate the above Toeplitz operator to a Wiener-Hopf operator on the real line. As originally discovered in \cite{Dev}, this can be done via  the Cayley transform. We will use the calculations in the paper \cite{CHO82} that are based on \cite{Dev}. 
These papers defined an isometry from $L^2$ of the real line to $L^2$ of the circle as follows.   Consider
the change of variable $\theta = 2\arctan x$ for $x\in\mathbb R$. Then define the isometry $W$ by $Wf(x) = f(2\arctan x)/(x-i)$. We can introduce two  orthonormal bases 
$$b_n(x) = e^{-2ni \arctan x}/(x-i)= (-1)^n (x+i)^n/(x-i)^{n+1}$$
and
$$b_{n+1/2}(x) = (-1)^{n+1/2}(x+i)^{n+1/2}/(x-i)^{n+3/2}$$
which correspond under $W$ to the bases $\{e_n\}_{n\in\Z}$ and $\{e_{n+1/2}\}_{n\in\Z}$ introduced above.
We  use them in two isometric copies of $L^2(\mathbb R)$ where the operator giving the isometry 
is $\tilde M_a$ which is multiplication by the function $x\to [(x+i)/(x-i)]^{1/2}$. Then as before we
 have the operator $\tilde M$ being the $2 \times 2$ matrix with $\tilde M_a$ in the off diagonal places which acts in the obvious fashion analogously to the definition above.

 The projection $Q$ maps to the projection $\tilde Q$ under $W$ and the properties of $\tilde Q \tilde M\tilde Q$ as a Fredholm operator of index $-1$ follow from those of $QMQ$ in Proposition \ref{prop_example} above.\footnote{
There is a way to think about this construction geometrically.
Taking square roots means considering a double cover of the circle or, to  incorporate the Wiener-Hopf case, working with a double cover of the Riemann sphere.
Moreover there is an analogue of the previous discussion where one takes an $n$-fold cover through considering $n^{th}$ roots. Then we use an $n\times n$ matrix of operators to play the role of $M$. This viewpoint is connected to an approach to the construction of anyons using a modification of the method in
\cite{CLa}. }
 \subsection{Preview}\label{sub_preview}
 
 We now briefly foreshadow the discussion in the next Section using the elementary example described above.
 If $u$ is a unitary matrix valued function on $\mathbb R$, say $u(x)= e^{-ig(x)}$ then associated to the pair $\frac{d}{idx}$,
$u(x)^*\frac{d}{idx} u(x) =\frac{d}{idx}+g'(x)$ there is an associated Witten index problem (explained in general in the next Section)
provided  the function $g'$ satisfies certain decay assumptions. Our results show that for this pair the Witten index
is given by $\frac{1}{2\pi}\int_{\mathbb R} g'(x)dx$. If $g'$ can be normalised so that $u(x)\to 1$ as $x\to\pm\infty$ then by the Gohberg-Krein theorem \cite{RD}, the operator $PM_uP$, where $P$ is the Hardy projection, is Fredholm. In this case, its Witten index is equal to the Fredholm index \cite{GS88}.
Starting from such a normalised function $g'$ we can scale by $\mu\in\mathbb R$ and setting $u_\mu(x)= e^{-i \mu g(x)}$ we find that the Witten index
for the pair $(\frac{d}{idx}$,
$u_\mu(x)^*\frac{d}{idx} u_\mu(x))$, or equivalently the pair $(\frac{d}{idx},\frac{d}{idx}+\mu g'(x))$, is $ \frac{\mu}{2\pi}\int_{\mathbb R}g'(x)dx$ implying that the Witten index is additive under composition
of unitaries $u_\mu$ and $u_{\mu'}$.

Thus in the example above where we start with $g$ being the arctan function we obtain the Witten index as $1/2$ so that multiplying two such functions gives
the Fredholm index $1$ of the corresponding Toeplitz operator.
 
 \section{A composition rule for the Witten index.}
 \label{sec_WI}

This Section contains the main result of this note. We show that the Witten index for a specific family of operators satisfies a composition formula similar to the Fredholm index. 
The exposition relies on notation and background from our earlier work \cite{CL}.

We introduce the setting for the upcoming discussion of the Witten index.

\begin{definition}Suppose that $A$ is a (in general,  unbounded) self-adjoint operator in a complex Hilbert space $\H$ and let $p\in \N\cup\{0\}$. A bounded self-adjoint operator $B$ on $\H$ is said to be 
\begin{enumerate}
\item a $p$-relative trace-class perturbation with respect to $A$ if $$B(A+i)^{-p-1}\in\L_1(\H).$$
\item $2p$-smooth with respect to $A$ if $B\in \bigcap_{k=1}^{2p}\dom(L_{A^2}^k),$ where the mapping $L_{A^2}^k$ is defined as follows 
\begin{equation}\label{def_L}
L_{A^2}^k(T)=\overline{(1+A^2)^{-k/2}[A^2,T]^{(k)}}
\end{equation} 
 with the domain
\begin{align*}
\dom(L_{A^2}^k)=\{ T\in\L(\H): T\dom(A^j)\subset \dom(A^j), \, j=1,\dots, 2k\\
\text{ 
and the operator } (1+A^2)^{-k/2}[A^2,T]^{(k)} \text{ defined on }\dom(A^{2k})\\
\text{ extends to a bounded operator on } \H\}.
\end{align*}
\end{enumerate}

\end{definition}

The main example we have to date \cite{CL}, \cite{CGL+22} starts with $A$ being
the flat $d$-dimensional space Dirac operator $\D$ acting in $L_2(\mathbb R^d) \otimes \mathbb C^{n(d)}, n(d)=2^{\lceil \frac{d}{2}\rceil},$ and the perturbation $B$ being  given by the multiplication operator by a smooth, $n(d) \times n(d)$ matrix-valued bounded function 
$$F:\mathbb R^d\to M^{n(d) \times n(d)}(L_{\infty}(\mathbb R) \cap C_b^{\infty}(\mathbb R)).$$
Here,  $\lceil \cdot \rceil$ is the  `ceiling function', that is, $\lceil x \rceil$ means the smallest integer larger than or equal to $x \in \R$.
Under suitable decay conditions at infinity for $F$, $B$  is a $p$-relative trace-class perturbation of $\D$  for $p>d$ and no smaller value of $p$. The smoothness assumption on $F$ guarantees that  the perturbation $B$ is $2p$-smooth with respect to $\D$ for any $p\in\N$.

To discuss the connection to index theory, we now fix a self-adjoint operator $A_1$ and assume that $B$ is a $2p$-smooth, $p$-relative trace-class perturbation with respect to $A_1$. We set $A_2=A_1+B$ which is also a self-adjoint operator, since $B$ is a bounded self-adjoint perturbation of $A_1$.
 We introduce   the `suspension' operator $\bsD_{A_1,A_2}$ for the pair $(A_1, A_2)$ as in \cite{APSIII}, \cite{RS95}.
Suppose that a positive function $\theta$ on $\R$ satisfies
\begin{align}\label{theta}
\begin{split}
\theta\in C_b^\infty(\R),& \quad \theta'\in L_1(\R),\\
\lim_{t\to-\infty}\theta(t)=0,&\quad \lim_{t\to+\infty}\theta(t)=1.
\end{split}
\end{align}

Denoting by $M_\theta$ the operator on $L_2(\mathbb R)$ of multiplication by $\theta$, we introduce the operator $\bsD_{A_1,A_2}$ in the Hilbert space $L_2(\mathbb R)\otimes \mathcal H$ by 
\begin{equation}\label{def_D_A}
\bsD_{A_1,A_2}=\frac{d}{dt}\otimes 1+ 1\otimes A_1 +M_\theta\otimes B.
\end{equation} 
Here the operator $d/dt$ in $L_2(\mathbb R)$  is the differentiation operator with domain being the Sobolev space $W^{1,2} \big(\mathbb R)$,
so that $\bsD_{A_1,A_2}$ is defined on $W^{1,2} \big(\mathbb R)\otimes \dom(A_1)$. For ease of notation we will usually identify 
$L_2(\mathbb R)\otimes \mathcal H$ with the Hilbert space  $L_2(\R,\H)$ of all $\H$-valued Bochner square-integrable functions on $\R$.

Next, we recall the definition of the (semigroup regularised) Witten index. Let $T$ be a closed, linear, densely defined operator in $\H$. 
Suppose that for some $t_0 > 0$  
\begin{equation*}
\big[e^{- t_0 T^* T} - e^{- t_0 TT^*}\big] \in \L_1(\H).   
\end{equation*} 
Then $\big(e^{-t T^*T} - e^{-t TT^*}\big) \in \L_1(\H)$ for all $t >t_0$ (see e.g. \cite[Lemma 3.1]{CGPST_Witten}).
The (semigroup regularized) \emph{Witten index} $W_s(T)$ of $T$ is defined by  
\begin{equation}\label{def_WI}
W_s(T) = \lim_{t \uparrow \infty}\tr\big(e^{-t T^*T} - e^{-t TT^*}\big) ,   
\end{equation}
whenever this limit exists.

The main result of \cite{CLPS} (see also \cite[Theorem 6.2.2]{CL}) is the so-called principal trace formula, which holds under the assumptions on $B$ as above.  Namely,  assuming that $B$ is  a $2p$-smooth, $p$-relative trace-class perturbation with respect to $A_1$, for any $t>0$ we proved the following relation:
\begin{align}\label{ch_principla trace formula_intro_formula}
\tr\Big(e^{-t\bsD_{A_1,A_2} \bsD_{A_1,A_2}^{*}}-e^{-t\bsD_{A_1,A_2}^{*} \bsD_{A_1,A_2}^{}}\Big)=-\Big(\frac{t}{\pi}\Big)^{1/2}\int_1^2\tr\Big(e^{-tA_s^2}B\Big)ds,\\
 \quad A_s=A_1+(s-1)B, \quad s\in[1,2],\nonumber
\end{align}
noting that our hypotheses guarantee both sides of the relation are well-defined. Since the right-hand side of the principal trace formula \eqref{ch_principla trace formula_intro_formula} depends only on the endpoints $A_1$ and $A_2=A_1+B$, it follows that (if it exists) the Witten index of the operator $\bsD_{A_1,A_2}$ depends only on the endpoints $A_1$ and $A_2$ and does not depend on the choice of the connection function $\theta$ in \eqref{theta}. For this reason, we introduce the following notion. 

\begin{definition}\label{Witten_index_pair}
Let $A_i,$ $i=1,2$, be self-adjoint operators in $\H$ with common dense domain. Let  $B:=A_2-A_1$ be  a $2p$-smooth, $p$-relative trace-class perturbation with respect to $A_1$ for some $p\in\N\cup\{0\}$. The \emph{Witten index for the pair $(A_1,A_2)$} (denoted by $W(A_1,A_2)$) is defined as
\begin{align*}
 W(A_1,A_2)&= W_s(\bsD_{A_1,A_2})
\end{align*}
where the operator $\bsD_{A_1,A_2}$ is defined as in \eqref{def_D_A}. 
\end{definition}
By \eqref{ch_principla trace formula_intro_formula} under the same assumption as in Definition \ref{Witten_index_pair},  we have that 
\begin{align}\label{WI_pair_ptf}
 W(A_1,A_2)&=-\lim_{t\to\infty}\Big(\frac{t}{\pi}\Big)^{1/2}\int_1^2\tr\Big(e^{-tA_s^2}B\Big)ds, 
\end{align}
where $A_s=A_1+(s-1)B, \quad s\in[1,2].$

\begin{remark}
Note that for $p=0$ or $p=1$, the principal trace formula \eqref{ch_principla trace formula_intro_formula} holds without the assumption of $2p$-smoothness of the perturbation $B$.  Indeed, this assumption is necessary only in the proof of \cite[Theorem 3.2.6]{CL} for convergence of the difference of higher powers of resolvents of the operators $\bsD_{A_1,A_2}^{}\bsD_{A_1,A_2}^{*}$ and $\bsD_{A_1,A_2}^{*} \bsD_{A_1,A_2}^{}$. For $p=0$ the principal trace formula was proved in \cite[Example B.6 (ii)]{CGPST_Witten}. For $p=1$, one can follow an argument similar to the proof of \cite[Proposition 2.4]{GLMST} to obtain  \cite[Theorem 3.2.6]{CL} for the difference of the resolvents of $\bsD_{A_1,A_2}^{}\bsD_{A_1,A_2}^{*}$ and $\bsD_{A_1,A_2}^{*} \bsD_{A_1,A_2}^{}$.
\end{remark}

Next we present the main result of this Section.
\begin{proposition}\label{prop_WI_sum}
Assume that $A_i$, $i=1,2,3$, are self-adjoint operators in a complex separable Hilbert space $\H$ with common dense domain and let $p\in\N\cup\{0\}$.  Assume that $A_2-A_1$ (respectively, $A_3-A_2$) is a $2p$-smooth, $p$-relative trace-class perturbation with respect to $A_1$ (respectively, with respect to $A_2$). If Witten indices for the pairs $(A_1,A_2)$ and $(A_2,A_3)$ exist, then the Witten index for the pair $(A_1,A_3)$ exists and 
$$W(A_1,A_2)+ W(A_2,A_3)=W(A_1, A_3),$$
or equivalently, 
$$W_s(\mathbf D_{A_1,A_2})+ W_s(\mathbf D_{A_2,A_3})=W_s(\mathbf D_{A_1, A_3}).$$
\end{proposition}
\begin{proof}
For simplicity we write $B_1=A_2-A_1$, $B_2=A_3-A_2$, and $B_3=A_3-A_1$. By assumption, $B_i$ is a $2p$-smooth, $p$-relative trace class perturbation with respect to $A_i$ for $i=1,2$. We claim that $B_3$ is also a $2p$-smooth, $p$-relative trace-class perturbation with respect to $A_1$.

We firstly show that $B_3(A_1-i)^{-p-1}$ is a trace class operator. Noting that $B_3=B_1+B_2$ we have 
\begin{align*}
B_3(A_1-i)^{-p-1}&=B_1(A_1-i)^{-p-1}+B_2(A_1-i)^{-p-1}\\
&=B_1(A_1-i)^{-p-1}+B_2(A_2-i)^{-p-1}\\
&\quad\quad+B_2\Big((A_1-i)^{-p-1}-(A_2-i)^{-p-1}\Big).
\end{align*} 
By assumption, the first and the second terms above are trace-class. By \cite[Theorem 3.1.7]{CL} the third term is also trace-class. Thus, $B_3$ is a $p$-relative trace-class perturbation with respect to $A_1$. 

By an argument similar to the proof of \cite[Proposition 2.9]{CGPRS}, one can show that $B_3$ is a $2p$-smooth perturbation with respect to $A_1$. By \eqref{WI_pair_ptf}, we have that
\begin{equation}\label{eq1}
W(A_1,A_3)=-\lim_{t\to\infty}\Big(\frac{t}{\pi}\Big)^{1/2}\int_1^2\tr\Big(e^{-t(A_1+(s-1)B_3)^2}B_3\Big)ds.
\end{equation} 
Since $B_1$ and $B_3$ are $p$-relative trace-class perturbations with respect to $A_1$, it follows that $B_2=B_3-B_1$ is also a $p$-relative trace-class perturbation of $A_1$. We can consider two paths joining $A_1$ and $A_3$ with the first path being $\{A_1+(s-1)B_3\}, s\in[1,2],$ and the second path consisting of two pieces $\{A_1+(s-1)B_1\}$, $s\in[1,2]$ and $\{A_1+B_1+(s-1)B_2\}, s\in[1,2]$. By \cite[Theorem 5.3.8]{CL} the integral on the right hand-side is independent of the choice of these two paths and 

$$\int_1^2\tr\Big(e^{-t(A_1+(s-1)B_3)^2}B_3\Big)ds$$
\begin{equation}\label{eq2}
=\int_1^2\tr\Big(e^{-t(A_1+(s-1)B_1)^2}B_1\Big)ds+\int_1^2\tr\Big(e^{-t(A_1+B_1+(s-1)B_2)^2}B_2\Big)ds.
\end{equation} 
By \eqref{WI_pair_ptf} we have that 
$$W(A_1,A_2)=-\lim_{t\to\infty}\Big(\frac{t}{\pi}\Big)^{1/2}\int_1^2\tr\Big(e^{-t(A_1+(s-1)B_1)^2}B_1\Big)ds$$
and 
$$W(A_2,A_3)=-\lim_{t\to\infty}\Big(\frac{t}{\pi}\Big)^{1/2}\int_1^2\tr\Big(e^{-t(A_2+(s-1)B_2)^2}B_2\Big)ds.$$
Combining these two equalities with \eqref{eq1} and \eqref{eq2} and recalling that $A_1+B_1=A_2$, we conclude that 
$$W(A_1,A_2)+ W(A_2,A_3)=W(A_1, A_3).$$
\end{proof}

\subsection{A one-dimensional example of the additivity property  of the Witten index}

We now specialise to the situation that arises in the scattering example of the final Section below.
In our paper \cite{Review} on trace formulas for a $1+1$-dimensional operator we considered the pair $A_1=\frac{d}{idx}$
as a self-adjoint operator  on $L^2(\mathbb R,  \mathbb C^d)$ and $A_2= UA_1U^* =A_1+\Phi$ where $U$ is 
a unitary operator on $L^2(\mathbb R,  \mathbb C^d)$ and where $\Phi$ is a $d\times d$ matrix valued function that satisfies with respect to $A_1$ a $p$-relative trace-class perturbation for $p=1$, that is $\Phi(1+A_1^2)^{-1}$ is trace class.
If $U$ is given by a function $U(x)=\exp{i\int_{x_0}^x \Phi(y) dy}$ we call this the commutative case. 
Note that in the  general case studied in \cite{Review}, the perturbation $\Phi$ is such that
we construct a unitary $U$  that conjugates $A_1$ to $A_2$ using the so-called Dyson expansion.  Here we consider only the case where this construction collapses
 to this simple exponential form.

We showed in \cite{Review} that the Witten index $W(A_1, U^*A_1U)$ associated to this pair exists and is given by
\begin{equation}\label{WI-integral}
W(A_1, U^*A_1U)=\int _{\mathbb R}\tr(U^*[A_1,U] )=\frac{1}{2\pi} \int_{\mathbb R} \tr(\Phi(x))dx.
\end{equation}
 
 Assume $V$ is a second unitary operator of the same form as $U$ associated with a function $\Psi$. Then, using again \cite{Review} we have that 
 $$W(A_1, V^*A_1V)=\int _{\mathbb R}\tr(V^*[A_1,V] )=\frac{1}{2\pi} \int_{\mathbb R} \tr(\Psi(x))dx.$$   
 
 Consider now the pair $A_1, V^*U^*A_1UV$. The elementary algebraic identity
 \begin{align*}
 \tr (V^*U^*[A_1, UV])&= \tr(V^*U^* [A_1, U]V + V^*U^*U[A_1,V])\\
 &= \tr(U^*[A_1,U]) +\tr(V^*[A_1,V])
 \end{align*}
shows that 
$$W(A_1, V^*U^*A_1UV)=W(A_1, U^*A_1U)+W(A_1, V^*A_1V).$$
This gives a very simple proof of an addition formula for Witten indices specialising that  in Proposition \ref{prop_WI_sum} in this one dimensional situation.

\section{An application for Levinson's theorem}
\label{sec_Levinson}

We describe in this Section an application of the additivity result for the Witten index to a question about Levinson's theorem for Schr\"odinger operators raised in \cite{AR}. We refer the reader to \cite{Yafaev} for basics of scattering theory relevant to our discussion below. 
For  the one dimensional Schr\"odinger scattering considered in this Section we refer the reader to \cite[Chapter 5]{Yafaev_II} and \cite{BGW} for further details. 

Let $H_0=-\frac{d^2}{dx^2}$ be the Laplacian on the real line $\R$ and let the interacting Hamiltonian be $H:=H_0+V$ with $V$ a fast decreasing potential function on
$\mathbb R$.  Then the scattering operator $S$ in the spectral representation of $H_0$ is a $2 \times 2$ unitary matrix valued function of the spectral variable $\lambda\in \R^+:=[0,\infty)$ (see \cite[Section 5.1.4]{Yafaev_II}).
In \cite{Levinson} Levinson showed that the number of  eigenvalues $N$ (counted with multiplicity) of the operator $H$  was related to the scattering operator $S$ by the formula
\begin{equation}\label{eq_Levinson}
N=\frac1\pi\Big(\delta(\infty)-\delta(0)\Big)+\frac12 M_R(0),
\end{equation}
where $\delta$ is given by the argument of the determinant of $S$ and $M_R(0)$ takes either the value zero or one.  The correction term $M_R(0)$ corresponds to the existence of resonancnes of $H$, which are  the solutions
the equation $H\psi=0$ with $\psi\notin L_2(\R).$

Following ideas of \cite{KR} the paper \cite{AR} formulates Levinson's  theorem as an index computation for a generalised Toeplitz operator (in the sense of \cite{DS}) constructed from $S$.
They observe that this formulation is possible firstly because $S-1$ as a function of the spectral variable $\lambda$ converges to zero as $\lambda\to\infty$ 
and secondly because they restrict to the case where $S(0)$ is the identity operator.

We will now show how the latter restriction can be removed via the use of the Witten index and the additivity formula in Proposition \ref{prop_WI_sum}.
 In order to use our results on the Witten index in this application we need to transform the generalised Toeplitz 
operator of \cite{AR} using an isometry of $L^2(\mathbb R^+, \mathbb C^2)$ with $L^2(\mathbb R, \mathbb C^2)$ defined by the exponential function $\exp: \mathbb R\to \mathbb R^+$.
This isometry takes the generator of dilations on $L^2(\mathbb R^+, \mathbb C^2)$ (which is used in \cite{KR} to give an expression for the wave operator) to the generator of translations, namely $d/id\lambda$ on $L^2(\mathbb R, \mathbb C^2)$. This reformulation of the situation in \cite{KR} on $\mathbb R$ then enables us to connect with our results on the Witten index.

Now, in the paper \cite{AR}, Levinson's theorem is formulated in such a way as to identify the number of bound states for $H_0+V$ as equal to the index of the generalised
Toeplitz operator $PSP$ where $P$ is the Hardy projection on $L^2(\mathbb R, \mathbb C^2)$ and we also use $S$ to denote the image of the scattering operator under the isometry defined via the exponential function on this latter Hilbert space.
The Hardy projection arises naturally in this formulation as it constructs the phase $F=2P-1$ of the operator $D=d/id\lambda$. 

There is a problem identified in \cite{AR} with their approach to Levinson's theorem when the scattering operator is not the identity at $0$ in the spectral representation of $H_0$.  To describe the problem
 we note that the exponential map takes $-\infty$ to $0$.  Thus in the formulation on $\mathbb R$ the difficulty presents itself as the failure of $S(\lambda)-1$ to converge to $0$ as the spectral variable 
 $\lambda\to -\infty$. When this happens it is well known that $PSP$ is not Fredholm.

With the notation as above we now formulate  Levinson's theorem as an index theorem under no restriction on the behaviour of the scattering matrix at $-\infty$. 
\begin{proposition}There exists a $2\times 2$ matrix valued function $\sigma$
such that the operator $PS\sigma^* P$ is Fredholm, with $P$ the Hardy projection. The number of bound states of the operator $H_0+V$ is equal to the Fredholm index
of $PS\sigma^* P$, which is equal to the sum of the Witten indices $W(D,S^*DS)$ and $W(D, \sigma^*D\sigma)$ where 
$$W(D, \sigma^*D\sigma)=\begin{cases}0,& \det(S(-\infty))=1,\\
                                                                                                                                                                                                                                                                                                                                                                                                                                                                                                                                                                                                                                                                                                                                                                                                                                                                                                                                                                                                                                                                                                                                                                                                                                                                                                                                                                                                                                                                                                                                                                                                                                                                                                                                                                                                                                                                                                                                                                                                                                                                                                                                                                                                                                                                                                                                                                                                                                                                                                                                                                                                                                                                                                                                                                                                                                                                                                                                                                                                                                                                                                                                                                                                                                                                                                                                                                                                                                                                                                                                                                                                                                                                                                                                                                                                                                                                                                                                                                                                                                                                                                                                                                                                                    1/2, &\text{otherwise}\end{cases}.$$
\end{proposition}
\begin{proof}
We start with  the Witten index $W(D, S^*DS)$ for the pair $(D,S^*DS)$. It is claimed \cite{ARen} that $S^*DS-D=S^*[D,S]$ is a $p$-relative trace-class perturbation of $D$ with $p=1$. Hence, by \eqref{WI-integral} we obtain 
the formula $W(D, S^*DS)=\frac{1}{2\pi i}\int_{\mathbb R}\tr(S^*[D,S])$.
 Moreover, \cite[equation (16)]{KR} shows that the Witten index $$W(D, S^*DS)=\frac{1}{2\pi i}\int_{\mathbb R}\tr(S^*[D,S])$$ may differ from an integer (which is the number of bound states of $H_0+V$) by $1/2$ depending on the determinant of $S(-\infty)$. 

Next, we choose $\sigma$  to have limit $S(-\infty)$ as $\lambda\to -\infty$ when $S(-\infty)$ is not the identity and 
to have the identity matrix as its limit as $\lambda\to \infty$.  Then, the  operator $S\sigma^*$ satisfies that  $S\sigma^*(\lambda)\to 1$ as  $\lambda\to \pm\infty$ and is continuous in $\lambda$.
The analysis of  \cite{AR} applies to $S\sigma^*$ and proves that the 
 operator $PS\sigma^*P$ is Fredholm and its index can be related to the number of bound states of $H_0+V$. 
 
By exploiting an index formula from noncommutative geometry  \cite{ARen, CGPRS, CGRS}, we have that 
$${\rm index}(PS\sigma^*P)=\frac{1}{2\pi i}\int_{\mathbb R} \tr(\sigma S^*[D,S\sigma^*]).$$
By \eqref{WI-integral} the right-hand side is the Witten index for the pair $(D, \sigma SDS\sigma^*)$.
Using the  commutator identity from the previous Section we write this as
\begin{align*}
{\rm index}(PS\sigma^*P)&=\frac{1}{2\pi i}\int_{\mathbb R} \tr( S^*[D,S]) + \frac{1}{2\pi i}\int_{\mathbb R} \tr(\sigma [D,\sigma^*])\\
&=W(D, S^*DS)+W(D, \sigma D\sigma^*).
\end{align*}

It remains to show that  $\sigma$ may be chosen so that  the Witten index $W(D,\sigma D\sigma^*)$ is either $0$ or $\frac12$ depending on the determinant of $S(-\infty)$.

Assume firstly, that $\det(S(-\infty))\neq 1.$ Then by \cite[Proposition 9]{KR}, we have that $\det(S(-\infty))=-1$ and  $S(-\infty)=\pm\begin{pmatrix}
1&0\\0&-1
\end{pmatrix}
$.  In this case,  we can set
$$
	\sigma(\lambda)= \pm\begin{pmatrix} 
	1& 0 \\
	0& e^{i(\arctan \lambda -\pi/2)}\\
	\end{pmatrix}, 
	$$
which interpolates between $-1$ at $-\infty$ and $1$ at $\infty$. We can then write $\sigma(\lambda)=\exp(-i \int_\lambda^\infty \begin{pmatrix}
0&0\\0& (1+x^2)^{-1}
\end{pmatrix}dx), \lambda\in \R.$
By \eqref{WI-integral} we have that 
$$W(D, \sigma D\sigma^*)=\frac1{2\pi} \int_\R \tr\begin{pmatrix}
0&0\\0& (1+\lambda^2)^{-1}\end{pmatrix} d\lambda=\frac12.$$

Next, assume that $\det(S(-\infty))=1$. Since $S(-\infty)$ is unitary, it follows that eigenvalues of $S(-\infty)$ are $e^{i\theta}$ and $e^{-i\theta}$ for some $\theta \in [0,2\pi]$. Let $u$ be a unitary matrix,  such that $S(-\infty)=u\begin{pmatrix}
 e^{-i\theta}&0\\0&e^{i\theta}
\end{pmatrix}u^*.$
We define 
$$\sigma(\lambda)=u\begin{pmatrix}e^{i\frac{\theta}{\pi}(\arctan \lambda -\pi/2)}&0\\0&e^{-i\frac{\theta}{\pi}(\arctan \lambda -\pi/2)}
\end{pmatrix}u^*.$$
 Then $\sigma(\pm \infty)=S(\pm \infty)$ and using
equality \eqref{WI-integral} as above we have 
 \begin{align*}
 W(D,\sigma D\sigma^*)=\frac1{2\pi}\int_\R \tr\begin{pmatrix}\frac{\theta}{\pi(\lambda^2+1)}&0\\0&-\frac{\theta}{\pi(\lambda^2+1)} 
 \end{pmatrix}d\lambda=0,
 \end{align*}
  as required. 
\end{proof}

Our introduction of $S\sigma^*$ as a `modified scattering operator'  
	links the Witten index in a special case to \cite{CGPRS}  though exactly how to interpret the  `trick' of introducing a correction factor via $\sigma$
	to compensate for the fact that $S(0)\neq 1$  is currently ad hoc
	and without a theoretical foundation.	
	
	In this Section we have only considered one dimensional scattering for Schr\"odinger operators.  For higher dimensions analogous results require
	an extension of what is currently known about the Witten index. Specifically, the scattering operator $S$ for the higher dimensional case, as a function of the spectral variable for $H_0$, is a unitary operator valued function with $S(\lambda)-1$ in a Schatten class (depending on the dimensionality) when there are no resonances.
	For this case we need to extend the results in \cite{Review}. As our proof of a formula for the Witten index in \cite{Review}  requires us to compute a formula for the spectral shift function
	for the pair $A_-, SA_-S^*$ this extension is likely to be quite difficult.

\end{document}